\documentclass[11pt]{amsart}
\usepackage{amssymb,amscd,  latexsym, graphicx, mathrsfs, enumerate}

\newcommand{\nc}{\newcommand}
\numberwithin{equation}{section}
\newtheorem{theorem}{Theorem}[section]
\newtheorem{prop}[theorem]{Proposition}
\newtheorem{importnota}[theorem]{Important Notation}
\newtheorem{prblm}[theorem]{Problem}
\newtheorem{notation}[theorem]{Notation}
\newtheorem{caution}[theorem]{Caution}
\newtheorem{remark}[theorem]{Remark}
\newtheorem{lemma}[theorem]{Lemma}
\newtheorem{construction}[theorem]{Construction}
\newtheorem{corollary}[theorem]{Corollary}
\newtheorem{example}[theorem]{Example}
\newtheorem{conclusion}[theorem]{Conclusion}
\newtheorem{triviality}[theorem]{Triviality}
\newtheorem{proto}[theorem]{Prototype Quasifibration}
\newtheorem{cauex}[theorem]{Cautionary Example}
\newtheorem{propositiondef}[theorem]{Proposition-Definition}
\newtheorem{subth}{Nuisance}[theorem]
\newtheorem{ssubth}{ }[subth]
\newtheorem{conjecture}[theorem]{Conjecture}
\newtheorem{sidest}[theorem]{Side Story}
\newtheorem{miniexample}[theorem]{Example}
\theoremstyle{definition}
\newtheorem{defin}[theorem]{Definition}

\nc\tri[1]{\begin{triviality}}
\nc\side[1]{\begin{sidest}}
\nc\conj[1]{\begin{conjecture}}
\nc\prodef[1]{\begin{propositiondef}}
\nc\prt[1]{\begin{proto}}
\nc\lem[1]{\begin{lemma}}
\nc\sblm[1]{\begin{sublemma}}
\nc\pro[1]{\begin{prop}}
\nc\thm[1]{\begin{theorem}}
\nc\cor[1]{\begin{corollary}}
\nc\dfn[1]{\begin{defin}}
\nc\sthm[1]{\begin{subth}}
\nc\exm[1]{\begin{example}}
\nc\miniexm[1]{\begin{miniexample}}
\nc\plm[1]{\begin{prblm}}
\nc\rmk[1]{\begin{remark}}
\nc\subrmk[1]{\begin{subremark}}
\nc\ntn[1]{\begin{notation}}
\nc\cau[1]{\begin{caution}}
\nc\imn[1]{\begin{importnota}}
\nc\cax[1]{\begin{cauex}}
\nc\con[1]{\begin{construction}}
\nc\ssthm[1]{\begin{ssubth}}
\nc\cnc[1]{\begin{conclusion}}
\nc\elem{\end{lemma}}
\nc\esblm{\end{sublemma}}
\nc\eside{\end{sidest}}
\nc\econj{\end{conjecture}}
\nc\eprodef{\end{propositiondef}}
\nc\eprt{\end{proto}}
\nc\ethm{\end{theorem}}
\nc\ecor{\end{corollary}}
\nc\edfn{\end{defin}}
\nc\esthm{\end{subth}}
\nc\epro{\end{prop}}
\nc\etri{\end{triviality}}
\nc\eexm{\end{example}}
\nc\eminiexm{\end{miniexample}}
\nc\ermk{\end{remark}}
\nc\subermk{\end{subremark}}
\nc\eplm{\end{prblm}}
\nc\ecau{\end{caution}}
\nc\ecax{\end{cauex}}
\nc\eimn{\end{importnota}}
\nc\entn{\end{notation}}
\nc\econ{\end{construction}}
\nc\ecnc{\end{conclusion}}
\nc\essthm{\end{ssubth}}

\newcommand{\R}{\mathbb{R}}
\newcommand{\Q}{\mathbb{Q}}

\newcommand{\Z}{\mathbb{Z}}

\newcommand{\A}{\mathbb{A}}

\newcommand{\G}{\Gamma}

\newcommand{\ds}{\displaystyle}

\newcommand{\bs}{\backslash}

\renewcommand{\Bbb}{\mathbb}

\title[Non-vanishing of Miyawaki type lift]
{Non-vanishing of Miyawaki type lift}
\author{Henry H. Kim and Takuya Yamauchi}
\thanks{The first author is partially supported by NSERC. The second author
is partially supported by JSPS Grant-in-Aid for Scientific Research (C) No.15K04787}
\address{Henry H. Kim \\
Department of mathematics \\
 University of Toronto \\
Toronto, Ontario M5S 2E4, CANADA \\
and Korea Institute for Advanced Study, Seoul, KOREA}
\email{henrykim@math.toronto.edu}
\address{Takuya Yamauchi \\ 
Mathematical Inst. Tohoku Univ.\\
 6-3,Aoba, Aramaki, Aoba-Ku, Sendai 980-8578, JAPAN and 
Max planck institute for mathematics, Bonn, 
Germany}
\email{tyamauchi@m.tohoku.ac.jp}

\begin{document}

\begin{abstract}

In this paper, we show the non-vanishing of the Miyawaki type lift for $GSpin(2,10)$ constructed in \cite{KY1}, by using the fact that the Fourier coefficient at the identity is closely related to the Rankin-Selberg $L$-function of two elliptic cusp forms. 
In the case of the original Miyawaki lift of Siegel cusp forms, we reduce the non-vanishing problem to that of the Rankin-Selberg convolution of two Siegel cusp forms.
\end{abstract}
\maketitle

\section{Introduction}

Miyawaki type lifts are kinds of Langlands functorial lifts and a special case was first conjectured by Miyawaki \cite{M} and proved by Ikeda for Siegel cusp forms in \cite{Ik1}. Since then, 
such a lift for Hermitian modular forms was constructed by Atobe and Kojima \cite{AK}, and for half-integral weight Siegel cusp forms by Hayashida \cite{H}, and we constructed Miyawaki type lift for ${\rm GSpin}(2,10)$ \cite{KY1}. 
Recently Ikeda and Yamana \cite{IY} generalized Ikeda type construction for Hilbert-Siegel cusp forms in a remarkable way and accordingly Miyawaki type lift for Hilbert-Siegel cusp forms for any level follows.   
 In all these works, a construction of Miyawaki type lift takes two steps as follows: First, construct Ikeda type lift on a bigger group from an elliptic cusp form, and then define a certain integral on a block diagonal element which is 
an analogue of pull-back formula studied by Garrett \cite{Garrett} for Siegel Eisenstein series. If the integral is non-vanishing, then it is shown that it is a Hecke eigen cusp form, and it is the Miyawaki type lift. The question of non-vanishing of the integral was left open. 

In this paper, we show the non-vanishing for certain special cases. The idea is to write the Ikeda type lift as Fourier-Jacobi expansion with matrix index $S$. Then the Fourier coefficients of the Miyawaki type lift become the integral of vector-valued modular forms and theta series. By choosing $S$ carefully, we can show that the Fourier coefficient of index $S$ is non-vanishing. 

In particular, in Section \ref{GS}, we show it for the Miyawaki type lift for 
$GSpin(2,10)$. The Miyawaki type lift in this case is a cusp form on $GSpin(2,10)$ associated to two cusp forms $f\in S_{2k}(SL_2(\Bbb Z))$ and $g\in S_{2k+8}(SL_2(\Bbb Z))$. 
In this case, the situation is very nice in that $S={\bold 1}_2$ is associated to an even unimodular matrix $E_8^{\oplus2}$ ($E_8$ denotes the unique even $8\times 8$ unimodular matrix), and the integral becomes essentially the Rankin-Selberg $L$-function $L(s,f\otimes g)$ at $s=4$. Therefore it is non-vanishing. Recently, we constructed the Ikeda type lift for the exceptional group of type $E_7$ for any level, and hence the Miyawaki type lift can be generalized in an obvious way. For the non-vanishing for higher level case, we use the adelic language by following \cite{T}.

In Section \ref{Ikeda}, we consider the original Miyawaki lift in \cite{Ik1}. Namely, 
the Miyawaki lift is a cusp form on 
$S_{k+n+r}(Sp_{4n+2r}(\Bbb Z))$ associated to two cusp forms $f\in S_{2k}(SL_2(\Bbb Z))$ and $g\in S_{k+n+r}(Sp_{2r}(\Bbb Z))$. 
Consider the Fourier-Jacobi expansion of the Ikeda lift $F_f\in S_{k+n+r}(Sp_{4n+4r}(\Bbb Z))$ 
with matrix index $S$, where $S$ is a half-integral symmetric matrix of size $2n+r$. If 
$2n+r$ is divisible by 8, there exists an even unimodular matrix of size $2n+r$, and in that case, the integral becomes the integral of two Siegel cusp forms and theta series. Here we use the Siegel's formula that says that the linear combination of the theta series is the Eisenstein series. Hence a linear combination of the integrals becomes the Rankin-Selberg convolution of two Siegel cusp forms \cite{Y}. It is not known whether the convolution is non-vanishing. If we assume the non-vanishing of the convolution, then the Miyawaki lift is non-vanishing. 

In Section \ref{Unitary}, we consider the Miyawaki lift for the unitary group in \cite{AK}. 
Namely, let $K$ be an imaginary quadratic field with discriminant $-D$, and $\chi=\chi_D$ be the Dirichlet character corresponding to $K/\Bbb Q$.
Let $f$ be a normalized Hecke eigen cusp form belonging to $\begin{cases} S_{2k}(SL_2(\Bbb Z)), &\text{if $n$ odd}\\ S_{2k+1}(\Gamma_0(D),\chi), &\text{if $n$ even}\end{cases}$. Then given a cusp form $g$ of weight $2k+2[\frac n2]+2r$ on $U(r,r)$ defined over $K/\Bbb Q$, 
the Miyawaki lift $\mathcal F_{f,g}$ is a cusp form of weight $2k+2[\frac n2]+2r$ on $U(n+r,n+r)$ defined over $K/\Bbb Q$. Atobe and Kojima \cite{AK} showed that if $\mathcal F_{f,g}$ is non-vanishing, it is a Hecke eigen form.
We consider the special case $4|(n+r)$. In this case, we use classification of even unimodular matrices over imaginary quadratic fields in \cite{CR}, and use the Siegel-Weil formula for unitary groups \cite{Ich} in order to obtain the analogue of the Siegel formula for Hermitian theta series. If $r=1$, the Miyawaki lift $\mathcal F_{f,g}$ is non-vanishing. If $r>1$, assuming the non-vanishing of the Rankin-Selberg convolution, we show that the Miyawaki lift is non-vanishing.

In the last section, we give an outline of non-vanishing of the Miyawaki lift for half-integral weight Siegel cusp forms assuming the non-vanishing of a similar integral involving half-integral weight modular forms.

Finally we remark that recently, Atobe \cite{At} independently obtained non-vanishing of the original Miyawaki lift by a different method. He assumes Gan-Gross-Prasad conjecture \cite{GGP}. 

\medskip

\noindent\textbf{Acknowledgments.} We would like to thank H. Atobe, T. Ikeda, S.  Hayashida, and M. Tsuzuki for helpful discussions. 
Special thanks are given to H. Atobe for pointing out some mistakes in an earlier version and to M. Tsuzuki 
for guiding the second author on the computation in Section 2.2.

\section{Miyawaki type lift for $GSpin(2,10)$}\label{GS} 
In this section we prove the non-vanishing of the Miyawaki type lift constructed in \cite{KY1}. We showed that under the assumption of non-vanishing, it is a Hecke eigen cusp form.
After these works the authors generalized the main theorems in \cite{KY} and hence the Miyawaki type lift can be generalized in an obvious way.
However we treat the non-vanishing separately for level one and higher level cases because of the 
nature of the construction. 
 
\subsection{Level one}
Let $f\in S_{2k}(SL_2(\Bbb Z)), g\in S_{2k+8}(SL_2(\Bbb Z))$ be Hecke eigen cusp forms.
Let $\mathcal F_{f,g}$ be the Miyawaki lift constructed in \cite{KY1}. It is defined as an integral:

\begin{equation}\label{integral}
\mathcal{F}_{f,g}(Z)=\int_{SL_2(\Z)\backslash \Bbb H} F_f\begin{pmatrix} Z&0\\0&\tau\end{pmatrix} \overline{g(\tau)}
({\rm Im} \tau)^{2k+6}\, d\tau,
\end{equation}
where $F_f$ is the Ikeda type lift constructed in \cite{KY}.

In \cite[Section 8]{KY1}, we wrote it as
\begin{equation}\label{ep1}
\mathcal F_{f,g}(Z)=\sum_{S} A_S e^{2\pi i Tr(TS)},
\end{equation}
where $S\in \frak J_2(\Bbb Z)_+$ (see \cite[(8.1)]{KY1}), and
\begin{equation}\label{ep2}
A_S=\int_{SL_2(\Bbb Z)\backslash \Bbb H} F_S(\tau,0)\overline{g(\tau)} Im(\tau)^{2k+8} \, d^*\tau,
\end{equation}
where $d^*\tau=\frac {dx dy}{y^2}$ is the invariant measure on $\Bbb H$. (See Section 8 of \cite{KY1}.)

If $S={\bold 1}_2\in \frak J_2(\Bbb Z)_{+}$, then by \cite[Appendix]{KY} (see also the corrections in Section 2 of \cite{KY2}), 
the quadratic form $\sigma_S$ associated to $S$ is 
of type $E^{\oplus 2}_8$ and then  
we can show that $\Xi(S)=\{0\}$ in \cite[(9.3)]{KY}. 
Note that for the basis $\{\alpha_i\}_{i=0}^7$  defining the integral Cayley numbers (see Section 2 of \cite{KY}), 
$\sigma_S=V\perp V\simeq E_8\perp E_8$ where $V$ is the quadratic space over $\Z$ given by 
$$\left(
\begin{array}{cccccccc}
 2 & 0 & 0 & 0 & 0 & -1 & -1 & -1 \\
 0 & 2 & 0 & 0 & 1 & -1 & 1 & 0 \\
 0 & 0 & 2 & 0 & 1 & 0 & -1 & 1 \\
 0 & 0 & 0 & 2 & 1 & 1 & 0 & -1 \\
 0 & 1 & 1 & 1 & 2 & 0 & 0 & 0 \\
 -1 & -1 & 0 & 1 & 0 & 2 & 0 & 0 \\
 -1 & 1 & -1 & 0 & 0 & 0 & 2 & 0 \\
 -1 & 0 & 1 & -1 & 0 & 0 & 0 & 2
\end{array}
\right)$$
whose determinant is one. 
Then 
from the formula \cite[(9.4)]{KY}, 
$$F_{S,0}(\tau)=\sum_{N>0} N^{\frac {2k-9}2}\prod_p \tilde{f}^p_{S,N}(\alpha_p) e^{2\pi i N\tau}.
$$
By the formula in \cite{Ka}, if $S={\bold 1}_2$, 
$$f^p_{S,N}(X)=\frac {1-X^{v_p(N)+1}}{1-X}.
$$
Hence $\tilde{f}^p_{S,N}=X^{-v_p(N)}+X^{-v_p(N)+2}+\cdots+X^{v_p(N)}$. So $F_{S,0}(\tau)=f(\tau)$.
Therefore,
$$F_S(\tau,0)=f(\tau)\theta(\tau),
$$
where $\theta$ is a theta function in 16 variables, and hence a modular form of weight 8 with respect to $SL_2(\Bbb Z)$.
Since dim $S_{8}(SL_2(\Bbb Z))=1$, $\theta(\tau)=E_8(\tau)$.
Here
$$E_8(\tau)=\sum_{\gamma\in \Gamma_\infty\backslash SL_2(\Bbb Z)} j(\gamma,\tau)^{-8},
$$
where $\Gamma_\infty=\left\{ \begin{pmatrix} a&b\\0&d\end{pmatrix}\in SL_2(\Bbb Z)\right\}$, and $j(\gamma,\tau)=c\tau+d$ for $\gamma=\begin{pmatrix} a&b\\c&d\end{pmatrix}\in SL_2(\Bbb Z)$. Hence
$$A_{{\bold 1}_2}=\int_{SL_2(\Bbb Z)\backslash \Bbb H} 
f(\tau)\overline{g(\tau)} E_8(\tau)Im(\tau)^{2k+8} \, d^*\tau.
$$
Then by the usual unfolding method,
\begin{eqnarray*}
&& A_{{\bold 1}_2}=\int_{SL_2(\Bbb Z)\backslash \Bbb H} \sum_{\gamma\in \Gamma_\infty\backslash SL_2(\Bbb Z)}
f(\gamma \tau)\overline{g(\gamma \tau)} Im(\gamma\tau)^{2k+8} \, d^*\tau\\
&& \phantom{xxx}=\int_{\Gamma_\infty\backslash SL_2(\Bbb Z)} f(\tau)\overline{g(\tau)} y^{2k+8} d^*\tau \\
&& \phantom{xxx}=\int_0^\infty y^{2k+6} \left(\int_0^1 f(x+iy)\overline{g(x+iy)} dx\right) dy.
\end{eqnarray*}

Let $f(\tau)=\sum_{n=1}^\infty a(n) e^{2\pi i n \tau}$, and $g(\tau)=\sum_{n=1}^\infty b(n) e^{2\pi i n\tau}$. Then
$$
\int_0^1 f(x+iy)\overline{h(x+iy)} dx=\sum_{n=1}^\infty a(n)\overline{b(n)} e^{-4\pi n y}.
$$
Therefore,
\begin{eqnarray*}
 A_{{\bold 1}_2}=\sum_{n=1}^\infty a(n)\overline{b(n)} \int_0^\infty y^{2k+6} e^{-4\pi n y}\, dy =(4\pi)^{-2k-7}\Gamma(2k+7) \sum_{n=1}^\infty \frac {a(n)\overline{b(n)}}{n^{2k+7}}.
\end{eqnarray*}

Let $L(s,f\otimes \bar g)$ be the Rankin-Selberg $L$-function:
$$L(s,f\otimes \bar g)=\zeta(2s) \sum_{n=1}^\infty \frac {a(n)\overline{b(n)}}{n^{s+2k+3}}=\prod_p \prod_{i=1}^2\prod_{j=1}^2 (1-\alpha_{f,i}(p)\alpha_{g,j}(p)p^{-s})^{-1},
$$
where $\alpha_{f,i}(p), \alpha_{g,j}(p)$ are roots of 
$X^2-\frac {a(p)}{p^{k-\frac 12}} X+1=0$, $X^2-\frac {b(p)}{p^{k+\frac 72}} X+1=0$, resp. 

Hence
$$L(4,f\otimes \bar g)=\zeta(8) \sum_{n=1}^\infty \frac {a(n)\overline{b(n)}}{n^{2k+7}}.
$$

Now $L(s,f\otimes\bar g)$ converges absolutely for $Re(s)>1$ and it has the Euler product. Hence 
it is non-vanishing for $Re(s)>1$. Therefore $A_{{\bold 1}_2}$ is non-vanishing.
Hence we have proved

\begin{theorem} Let $\mathcal{F}_{f,g}$ be the Miyawaki type lift as above. Then it is non-zero.
\end{theorem}

\subsection{Higher level}  

Let $N,k$ be positive integers and 
$\G_0(N)=\left\{ \begin{pmatrix} a&b\\c&d\end{pmatrix}\in SL_2(\Bbb Z):\, c\equiv 0 (\text{mod $N$})\right\}$. 
We denote by $S_k(\G_0(N))$ the space of elliptic cusp forms of weight $k$ with respect to $\G_0(N)$. 
Throughout this section we keep this notation. 

Let $f\in S_{2k}(\G_0(N)), g\in S_{2k+8}(\G_0(N))$ be two newforms. 
Let $\varphi_f=\otimes'_p f_p,\ \varphi_g=\otimes'_p g_p$ be the decomposition of cuspidal automorphic forms 
associated to $f,g$ where each component at $p$ is chosen as a local newform defined in \cite{Schmidt} so that 
it takes the value 1 at the identity $I_2$ when $p$ is an unramified place. Let $\pi_{f}=\otimes'_p \pi_{f,p}, 
\pi_{g}=\otimes'_p \pi_{g,p}$ be the cuspidal automorphic representations generated by $\varphi_f,\varphi_g$ respectively. 
Applying Theorem 1.1 of \cite{KY2} with $\otimes_{p<\infty}\pi_{f,p}$ and using the classical interpretation (cf. Section 5.1 of \cite{KY}),  
we have the Ikeda type lift $F_f$. As in the case of level one, we can also define 
the Miyawaki type lift $\mathcal F_{f,g}$ by means of the integral given in \cite[(1.1)]{KY1} as follows:
\begin{equation}\label{integral}
\mathcal{F}_{f,g}(Z)=\int_{\G_0(N)\backslash \Bbb H} F_f\begin{pmatrix} Z&0\\0&\tau\end{pmatrix} \overline{g(\tau)}
({\rm Im} \tau)^{2k+6}\, d\tau.
\end{equation}
By the definition it is easy to see that 
$\mathcal{F}_{f,g}$ is a modular form for a congruence subgroup of $GSpin(2,10)(\Bbb Z)$.   

\begin{theorem}\label{gspin-nt} Keep the notations as above. 
Suppose that $N$ is square free. Then the Miyawaki type lift $\mathcal{F}_{f,g}$ is non-vanishing.
\end{theorem}

\begin{remark} As in \cite{KY1}, we may show that $\mathcal{F}_{f,g}$ is a Hecke eigen form and hence gives rise to a cuspidal representation of $GSpin(2,10)$.
\end{remark}

Since we have a non-trivial level, we should be careful with the fact that a priori we do not have an explicit form 
as in the case of level one. 
However by virtue of Lemma 7.1 and Lemma 7.3 of \cite{KY2}, we see that  
$$F_{\textbf 1_2}(\tau,0)=f(\tau)\theta(\tau)
$$
where $\theta(\tau)=E_8(\tau)$ is as before and the left hand side is defined similarly as in (\ref{ep1}),(\ref{ep2}). 
As in level one case, we have 
$$A_{{\bold 1}_2}=
\int_{\G_0(N)\backslash \Bbb H} f(\tau)\overline{g(\tau)} E_8(\tau){\rm Im}(\tau)^{2k+8} \, d^*\tau.
$$
Theorem \ref{gspin-nt} will follow from the following:
\begin{theorem}\label{rs-gspin1}
There exists a non-zero constant $C$ such that 
$$A_{{\bold 1}_2}=C\cdot L(4,\pi_f\times \pi_g)\neq 0,
$$
where $L(s,\pi_f\otimes\pi_g)$ is the Rankin-Selberg convolution for cuspidal representations $\pi_f,\pi_g$ 
attached to $f,g$ respectively. 
\end{theorem}

\begin{proof} We work on adelic forms. Let $K_p=GL_2(\Z_p)$ and 
$K_0(p^r)=\left\{\begin{pmatrix} a&b\\c&d\end{pmatrix}\in K_p :\, c\equiv 0 (\text{mod $p^r$})\right\}$. 

Let $\Phi_p$ be the local section in the principal series $\pi_p(|\cdot|^{\frac{7}{2}},|\cdot|^{-\frac{7}{2}})$ defined by, 
if $p\neq \infty$, 
$$\Phi_p(g)=\left\{
\begin{array}{cc}
|ad^{-1}|^4_p & {\rm if}\ 
g\in \left(\begin{array}{cc}
a & \ast \\
0 & d 
\end{array}\right)K_0(p^{{\rm ord}_p(N)}) \\
0 & {\rm otherwise.}
\end{array}
\right.
$$
Similarly we also define a spherical vector $\Phi^{{\rm ur}}_p$ by replacing $K_0(p^{{\rm ord}_p(N)})$ with $K_p$. 
If $p=\infty$, we define $\Phi_\infty(g)=\Phi^{{\rm ur}}_\infty(g)=
|ad^{-1}|^4 e^{\sqrt{-1}\theta k}$ if 
$g=pk_\theta\in \left(\begin{array}{cc}
a & \ast \\
0 & d 
\end{array}\right)K_\infty$, 
where $k_\theta=\left(\begin{array}{cc}
\cos \theta & -\sin \theta \\
\sin \theta & \cos \theta 
\end{array}\right)\in K_\infty\simeq SO(2)$. 

Put $\Phi^{\rm ur}=\otimes'_p \Phi^{\rm ur}_p$ and $\Phi=\otimes'_p \Phi_p$. The global section $\Phi^{\rm ur}$ gives rise to $E_8(\tau)$, the Eisenstein series of level one, while the global section $\Phi$ gives rise to the Eisenstein series $E_8^{(N)}(\tau)=\ds\sum_{\gamma\in \Gamma_\infty\backslash \G_0(N)} j(\gamma,\tau)^{-8}$ for $\G_0(N)$ with respect to 
the cusp $\infty$.

By the usual unfolding method, there exists a non-zero constant $C'$ depending on the normalization at all 
Steinberg places and at the infinite place such that  
\begin{eqnarray*}
&& A_{{\bold 1}_2}=C'\int_{Z(\A_\Q)GL_2(\Q)\bs GL_2(\A_\Q)}\sum_{\gamma\in B(\Q)\bs GL_2(\Q)}\Phi^{{\rm ur}}(\gamma g)\varphi_f(g)
\overline{\varphi_g(g)}dg   \\
&&\phantom{xxx}=C'\prod_p\int_{PGL_2(\Q_p)}\Phi^{{\rm ur}}_p(g)\varphi_{f,p}(g)\overline{\varphi_{g,p}(g)}dg. 
\end{eqnarray*}
For any unramified prime $p$, by Proposition 3.8.1 of \cite{Bump} we have 
$$\int_{PGL_2(\Q_p)}\Phi^{{\rm ur}}_p(g)\varphi_{f,p}(g)\overline{\varphi_{g,p}(g)}dg=\zeta_p(8)^{-1}L(4,\pi_{f,p}\times \pi_{g,p}).$$
For a bad place $p|N$ (hence $p||N$ by the assumption), by applying computation in the proof of Lemma 2.14 of \cite{T} 
(see line -7 through the bottom in page 22 of loc.cit.), we have 

\begin{eqnarray*}
&& \int_{`GL_2(\Q_p)}\Phi^{{\rm ur}}_p(g)\varphi_{f,p}(g)\overline{\varphi_{g,p}(g)}dg  \\
&&=\int_{PGL_2(\Q_p)}\Phi_p(g)\varphi_{f,p}(g)\overline{\varphi_{g,p}(g)}dg+
\int_{PGL_2(\Q_p)}\Phi_p(\left(\begin{array}{cc}
p & 0 \\
0 & 1 
\end{array}\right)g)\varphi_{f,p}(g)\overline{\varphi_{g,p}(g)}dg. \\
&&=[K_p,K_0(p)]^{-1}p^4 \zeta_p(8)^{-1}L(4,\pi_{f,p}\times \pi_{g,p}). 
\end{eqnarray*}

For $p=\infty$, it is well-known (cf. p.136 of \cite{KMV}) that 
$$C_\infty:=\int_{PGL_2(\R)}\Phi_\infty(g)\varphi_{f,\infty}(g)\overline{\varphi_{g,\infty}(g)}dg=\zeta_\infty(8)^{-1}
\Gamma(8)\Gamma(2k+7)=\Gamma(2k+7).
$$
Put $C=C'C_\infty\zeta^N(8)^{-1}[SL_2(\Z):\G_0(N)]^{-1}N^4$ where $\zeta^N(s)=\ds\prod_{p\nmid N}\zeta_p(s)$ is the 
partial Riemann zeta function outside $N$. Then we have $A_{\bold{1}_2}=C\cdot L(4,\pi_f\times \pi_g)$. 
Now $L(s,\pi_f\times \pi_g)$ converges absolutely for $Re(s)>1$ and it has the Euler product. Hence 
it is non-vanishing for $Re(s)>1$. Therefore $A_{{\bold 1}_2}$ is non-vanishing.
\end{proof}

\section{Miyawaki lift for Siegel cusp forms}\label{Ikeda}

In this section we assume that the readers are familiar with notations and results in \cite{Ik} and \cite{Ik1}. 

Let $h(\tau)\in S_{k+\frac 12}^+(\Gamma_0(4))$ be a Hecke eigenform in Kohnen's plus space corresponding to a Hecke eigenform $f(\tau)\in S_{2k}(SL_2(\Bbb Z))$.
Let $n, r$ be positive integers such that $n+r\equiv k$ mod 2. Then we have the Ikeda lift $F_f\in S_{k+n+r}(Sp_{4n+4r}(\Bbb Z))$ whose standard $L$-function is 
$$\zeta(s)\prod_{k=1}^{2n+2r} L(s+k+n+r-i,f).
$$
Now for $g\in S_{k+n+r}(Sp_{2r}(\Bbb Z))$, the Miyawaki lift is given by
$$\mathcal F_{f,g}(Z)=\int_{Sp_{2r}(\Bbb Z)\backslash \Bbb H_r} F_f\begin{pmatrix} Z&0\\0&W\end{pmatrix} \overline{g^c(W))} det(Im W)^{k+n-1}\, dW,
$$
where $Z\in \Bbb H_{2n+r}$, and $g^c(W)=\overline{g(-\overline{W})}$. Ikeda \cite{Ik1} showed that if the integral is non-vanishing, 
$\mathcal F_{f,g}$ is a Hecke eigenform in $S_{k+n+r}(Sp_{4n+2r}(\Bbb Z))$.
Now we have the Fourier-Jacobi expansion of $F_f$:
$$
F_f\begin{pmatrix} Z&u\\ {}^t u&W\end{pmatrix}=\sum_S \mathcal F_S(W,u) e^{2\pi i Tr(ZS)},
$$
where $S\in S_{2n+r}'(\Bbb Z)$, and $\mathcal F_S$ is a Fourier-Jacobi coefficient of index $S$. Then
$$\mathcal F_S(W,u)=\sum_{\lambda\in \Lambda} \theta_{[\lambda]}(S;W,u) \mathcal F_{S,\lambda}(W),
$$
where $\Lambda=(2S)^{-1}\Bbb Z^{2n+r}/\Bbb Z^{2n+r}$, and $\theta_{[\lambda]}(S;W,u)$ is a theta series of $2n+r$ variables.

Then
$$\mathcal F_{f,g}(Z)=\int_{Sp_{2r}(\Bbb Z)\backslash \Bbb H_r} \left(\sum_S \mathcal F_S(W,0) e^{2\pi i Tr(ZS)}\right)
\overline{g^c(W))} \det({\rm Im} W)^{k+n-1}\, dW
=\sum_S A_S e^{2\pi i Tr(ZS)},
$$
where 
$$A_S=\int_{Sp_{2r}(\Bbb Z)\backslash \Bbb H_r} \mathcal F_S(W,0)\overline{g^c(W)} det(Im W)^{k+n-1}\, dW.
$$
Now in order to interchange the sum and the integral, we need to show that the integral is absolutely convergent. We follow Lemma 7.1 of \cite{KY}: We have
$$\mathcal F_S(W,0) e^{-2\pi Tr(Y S)}=\int_{X} F_f\begin{pmatrix} Z&0\\ 0&W\end{pmatrix}\, e^{-2\pi i Tr(X S)}\, dX.
$$
Setting $Y=\frac 1{Tr(S)} I_{2n+r}$, we have $|\mathcal F_S(W,0)|\ll  (Im W)^{-(k+n+r)} Tr(S)^{2(k+n+r)}.$
Then for a fixed $Z$,
$$\sum_S |\mathcal F_S(W,0) e^{2\pi i Tr(Z S)}|\leq (Im W)^{-(k+n+r)} \sum_S Tr(S)^{2(k+n+r)} e^{-2\pi Tr(Y S)}.
$$
Now 
$Tr(YS)\geq c_Y Tr(S)$ for a constant $c_Y>0$. Hence 
$$
\sum_S Tr(S)^{2(k+n+r)} e^{-2\pi Tr(Y S)}\leq \sum_S Tr(S)^{2(k+n+r)} e^{-2\pi c_Y Tr(S)},
$$
which is bounded. Therefore,
$$\int_{Sp_{2r}(\Bbb Z)\backslash \Bbb H_r} \left(\sum_S \mathcal F_S(W,0) e^{2\pi i Tr(ZS)}\right)
\overline{g^c(W))} det(Im W)^{k+n-1}\, dW,
$$ 
converges absolutely.

Now we consider the special case: $r$ is even and $2n+r$ is a multiple of 8. 
Choose $S$ so that $2S$ is an even unimodular matrix. Then $\Lambda$ is trivial, and
$\mathcal F_S(W,0)=H(W)\theta_S(W)$, where $H(W)$ is a Siegel cusp form of weight $k+\frac r2$ and 
$\theta_S(W):=\theta_{[0]}(S;W,0)$ is the theta function for $S$ and the trivial class $[0]$, which is a Siegel modular form of weight 
$n+\frac r2$ and level one. By using Lemma 7.1 and 7.2 (see also Corollary 7.3-(2),(3)) and Lemma 7.7 of \cite{IY}, one can show that 
$H(W)$ is the Ikeda lift $F^{(r)}_f$ of $f$ to $Sp_{2r}$. 

Now we apply the Siegel formula.
\begin{theorem}[Siegel]
For $8| m$, $C_m$ be the set of classes (up to isomorphism) of even unimodular lattices of $m$ variables, and let $g_S\in C_m$ be the order of the automorphism group of $S\in C_m$. Let $E_{2k}(Z)=\sum_{(C,D)} det(CZ+D)^{-2k}$ be the usual Siegel Eisenstein series of weight $2k$.
 Then 
$$\sum_{S\in C_m} \frac 1{g_S} \theta_S=M_m E_{\frac m2},\quad \text{where $M_m=\sum_{S\in C_m} \frac 1{g_S}$}.
$$
\end{theorem}

Now for $C_{2n+r}$, consider the sum
\begin{equation}\label{A_S}
\sum_{S\in C_{2n+r}} \frac 1{g_S} A_S.
\end{equation}
It is
\begin{equation}\label{Siegel-f}
M_{2n+r}\int_{Sp_{2r}(\Bbb Z)\backslash \Bbb H_r} H(W)\overline{g^c(W)}E_{n+\frac r2}(W) det(Im W)^{k+n+r}\, d^*W,
\end{equation}
where $d^*W=det(Im W)^{-(r+1)}dW$.
Let 
$$H(W)=F^{(r)}_f(W)=\sum_{T\in S_r(\Bbb Z)^+} a_{F^{(r)}_f}(T) e^{2\pi i Tr(TW)},\quad g(W)=\sum_{T\in S_r(\Bbb Z)^+} a_g(T) e^{2\pi i Tr(TW)}.
$$
Then by \cite{Kal}, the integral (\ref{Siegel-f}) is exactly a scalar multiple of the Rankin convolution
$$R(k+n+\frac {r-1}2,H,\bar g)=
\sum_{T\in \widetilde S_r(\Bbb Z)^+} \frac {a_{F^{(r)}_f}(T)\overline{a_g(T)}}{\epsilon(T) det(T)^{k+n+\frac {r-1}2}},
$$
where $\widetilde S_r(\Bbb Z)^+$ is the set of $GL_r(\Bbb Z)$-equivalence classes of matrices $T\in S_r(\Bbb Z)^+$, and 
$\epsilon(T)=\#\{ U\in GL_r(\Bbb Z) | \, UT {}^t U=T\}$. By \cite[Lemma 3.1]{Y}, the above series converges absolutely if $2n>r+2$. However, we do not know the non-vanishing. We assume
\begin{conjecture} \label{rankin} $R(k+n+\frac {r-1}2,H,\bar g)$ is non-vanishing. 
\end{conjecture}

Under Conjecture \ref{rankin}, (\ref{A_S}) is non-vanishing. Then one of $A_S$ is non-vanishing. Hence we have proved:

\begin{theorem} Let $f(\tau)\in S_{2k}(SL_2(\Bbb Z))$, and $g\in S_{k+n+r}(Sp_{2r}(\Bbb Z))$ such that $2n+r$ is a multiple of 8. 
Then under Conjecture \ref{rankin}, the Miyawaki lift $\mathcal F_{f,g}$ is non-vanishing.
\end{theorem}

\section{Non-vanishing of Miyawaki lifts for $U(n,n)$}\label{Unitary}

We review the Miyawaki lift for the unitary group in \cite{AK}.
Let $K$ be an imaginary quadratic field with discriminant $-D$ and let $\mathcal O$ be the ring of integers. Let $\chi=\chi_D$ be the Dirichlet character corresponding to $K/\Bbb Q$.
Let $f$ be a normalized Hecke eigen cusp form belonging to $\begin{cases} S_{2k}(SL_2(\Bbb Z)), &\text{if $n$ odd}\\ S_{2k+1}(\Gamma_0(D),\chi), &\text{if $n$ even}\end{cases}$.

Ikeda \cite{Ik2} constructed a lift $I^{(n+2r)}(f)\in S_l(\Gamma_K^{(n+2r)}, \mbox{det}^{-\frac l2})$, 
where $l=2k+2[\frac n2]+2r$. It is a Hecke eigenform on the Hermitian upper half space $\mathcal H_{n+2r}$ of degree $n+2r$.

Now for a Hecke eigen cusp form $g\in S_l(\Gamma_K^{(r)},\mbox{det}^{-\frac l2})$, define
$$\mathcal F_{f,g}(Z)=\int_{\Gamma_K^{(r)}\backslash \mathcal H_r} I^{(n+2r)}(f)\begin{pmatrix} Z&0\\0&W\end{pmatrix} \overline{g^c(W)} det(Im W)^{l-2r}\, dW,
$$
for $Z\in \mathcal H_{n+r}$. (Here if $4|l$, $h\in S_{l}(SL_2(\Bbb Z))$ can be regarded as $h\in S_{l}(\Gamma_K^{(1)}, \mbox{det}^{-\frac l2})$. 
Note that $\Gamma_K^{(1)}=SL_2(\Bbb Z)\cdot \{\alpha\cdot {\bold 1}_2 | \, \alpha\in \mathcal O^\times\}$. \cite[page 1111]{Ik2})
Then $\mathcal F_{f,g}\in S_l(\Gamma_K^{(n+r)}, \mbox{det}^{-\frac l2})$. 
Atobe and Kojima \cite{AK} showed that if $\mathcal F_{f,g}$ is not identically zero, it is a Hecke eigen form, and its standard $L$-function is given by
$$L(s,\mathcal F_{f,g},St)=L(s,g,St)\prod_{i=1}^n L(s+\tfrac {n-1}2-i,\pi_f)L(s+\tfrac {n-1}2-i,\pi_f\otimes\chi).
$$

Consider the Fourier-Jacobi expansion
$$I^{(n+2r)}(f)\begin{pmatrix} Z&0\\0&W\end{pmatrix}=\sum_{S} \mathcal F_S(W,0) e^{2\pi i Tr(ZS)},
$$
where $S\in S_{n+r}'(\mathcal O)$, $(n+r)\times (n+r)$ positive definite semi-integral Hermitian matrices, 
and $\mathcal F_S$ is a Fourier-Jacobi coefficient of index $S$. 
Then
$$\mathcal F_S(W,0)=\sum_{\lambda\in \Lambda} \theta_{[\lambda]}(S;W,0) \mathcal F_{S,\lambda}(W),
$$
where $\Lambda=(2S)^{-1}\mathcal O^{n+r}/\mathcal O^{n+r}$, and $\theta_{[\lambda]}(S;W,0)$ is a theta series of $n+r$ variables.
Then

$$\mathcal F_{f,g}(Z)
=\sum_S A_S e^{2\pi i Tr(ZS)},
$$
where 
$$A_S=\int_{\Gamma_K^{(r)}\backslash \mathcal H_r} \mathcal F_S(W,0)\overline{g^c(W))} det(Im W)^{2k+2[\frac n2]}\, dW.
$$

Now we consider the special case: $4| (n+r)$. Choose $S$ so that $2S$ is an even unimodular Hermitian matrix. Then $\Lambda$ is trivial, and
$\mathcal F_S(W,0)=H(W)\theta_S(W)$, where $H(W)$ is a Hermitian cusp form of weight $2k+2[\frac n2]-n+r$ and 
$\theta_S(W):=\theta_{[0]}(S;W,0)$ is the theta function for $S$ and the trivial class $[0]$, which is a Hermitian modular form of weight 
$n+r$ and level one. 

For $l>2m$, let $E_{l}(Z):=\sum_{\G_{\infty}\bs \G_K^{(m)}}(\det g)^{\frac l2} \det(CZ+D)^{-l}$.
By applying the Siegel-Weil formula for unitary groups \cite{Ich}, we obtain the following 
analogue of the Siegel formula for Hermitian lattices:
\begin{theorem}
For $4|m$, let $C_m$ be the set of classes (up to isomorphism) of even unimodular Hermitian lattices of $m$ variables, and let $g_S\in C_m$ be the order of the automorphism group of $S\in C_m$. Then 
$$\sum_{S\in C_m} \frac 1{g_S} \theta_S=M_m E_{m},\quad \text{where $M_m=\sum_{S\in C_m} \frac 1{g_S}$}.
$$
\end{theorem}

Now for $C_{n+r}$, consider the sum
$$\sum_{S\in C_{2n+r}} \frac 1{g_S} A_S.
$$
It is
\begin{equation}\label{Siegel}
M_{2n+r}\int_{\Gamma_K^{(r)}\backslash \mathcal H_r} H(W)\overline{g^c(W)} E_{n+r}(W) det(Im W)^{2k+2[\frac n2]}\, dW.
\end{equation}
Let 
$$H(W)=\sum_{T\in S_r'(\mathcal O)} a_H(T) e^{2\pi i Tr(TW)},\quad g(W)=\sum_{T\in S_r'(\mathcal O)} a_g(T) e^{2\pi i Tr(TW)}.
$$
Now the integral (\ref{Siegel}) is exactly a scalar multiple of the Rankin convolution
$$R(k+n+r,H,\bar g)=\sum_{T\in \widetilde S'_{r}(\mathcal O)^+} \frac {a_H(T)\overline{a_g(T)}}{\epsilon(T) det(T)^{2k+2[\frac n2]+r}},
$$
where $\widetilde S'_{r}(\mathcal O)^+$ is the set of $GL_r(\mathcal O)$-equivalence classes of matrices $T\in S_r(\mathcal O)^+$, and 
$\epsilon(T)=\#\{ U\in GL_r(\mathcal O) | \, UT {}^t \bar U=T\}$. 

Now as in the Siegel case, we can show that $H(W)$ is the Ikeda lift of $f$ to $U(r,r)$. 
When $r=1$, $n$ is of the form $4m+3$, and the above integral is related to the Rankin-Selberg $L$-function $L(\frac {n+1}2, H\otimes \bar g)$ as in Section \ref{GS}. Hence it is non-vanishing.

If $r>1$, we assume the analogue of Conjecture \ref{rankin}. Then
\begin{theorem} 
Let $K$ be an imaginary quadratic field with discriminant $-D$. Let $\chi=\chi_D$ be the Dirichlet character corresponding to $K/\Bbb Q$.
Let $f$ be a normalized Hecke eigen cusp form belonging to $\begin{cases} S_{2k}(SL_2(\Bbb Z)), &\text{if $n$ odd}\\ S_{2k+1}(\Gamma_0(D),\chi), &\text{if $n$ even}\end{cases}$.
Let $l=2k+2[\frac n2]+2r$, and $g\in S_{l}(\Gamma_K^{(r)},\mbox{det}^{-\frac l2})$, and 
$\mathcal F_{f,g}$ be the Miyawaki lift in $S_{l}(\Gamma_K^{(n+r)},\mbox{det}^{-\frac l2})$. Let $4| (n+r)$. If $r=1$, $\mathcal F_{f,g}$ is non-vanishing. If $r>1$, assuming the analogue of Conjecture \ref{rankin}, it is non-vanishing.
\end{theorem}

\section{Miyawaki lift for half-integral Siegel cusp forms} \label{half}
In this section, assuming the non-vanishing of the integral (\ref{half-1}), we give an outline of non-vanishing of the Miyawaki lift for half-integral Siegel cusp forms given in \cite{H}. 
Let $I^{(2n)}(f)$ be the Ikeda lift as in Section \ref{Ikeda} for $f\in S_{2k}(SL_2(\Bbb Z))$. Consider its Fourier-Jacobi expansion with integer index:
$$
I^{(2n)}(f)\begin{pmatrix} Z_1&u\\ {}^t u&\tau\end{pmatrix}=\sum_{m=1}^\infty \psi_m(Z_1,u) e^{2\pi i m\tau},
$$
where $Z_1\in \Bbb H_{2n-1}$, and $\tau\in \Bbb H$.
Here $\psi_1(Z_1,u)$ is a Jacobi cusp form of weight $k+n$ and index 1 of degree $2n-1$. By the Eichler-Zagier-Ibukiyama correspondence \cite{Ib}, there exists a Siegel cusp form $F_f\in S_{k+n-\frac 12}^+(\Gamma_0^{(2n-1)}(4))$ which corresponds to $\psi_1$. For $g\in S_{k+n-\frac 12}^+(\Gamma_0(4))$, we put
$$\mathcal F_{f,g}(Z)=\int_{\Gamma_0(4)\backslash \Bbb H} F_f\begin{pmatrix} Z&0\\0&\tau\end{pmatrix} \overline{g(\tau)} Im(\tau)^{k+n-\frac 52}\, d\tau,
$$
for $Z\in \Bbb H_{2n-2}$. Then $\mathcal F_{f,g}$ is a cusp form in $S_{k+n-\frac 12}^+(\Gamma_0^{(2n-2)}(4))$.
Hayashida \cite{H} proved that if $\mathcal F_{f,g}$ is not identically zero, it is an eigenform with the standard $L$-function
$$L(s,\mathcal F_{f,g}, St)=L(s,g)\prod_{i=1}^{2n-3} L(s-i,h).
$$
Consider the Fourier-Jacobi expansion of $F_f$ with matrix index
$$F_f\begin{pmatrix} Z&0\\0&\tau\end{pmatrix}=\sum_{S} \mathcal F_S(\tau,0) e^{2\pi i Tr(ZS)},
$$
where $S\in S_{2n-2}'(\Bbb Z)$, and $\mathcal F_S$ is a Fourier-Jacobi coefficient of index $S$. 
Then
$$\mathcal F_S(\tau,0)=\sum_{\lambda\in \Lambda} \theta_{[\lambda]}(S;\tau,0) \mathcal F_{S,\lambda}(\tau),
$$
where $\Lambda=(2S)^{-1}\Bbb Z^{2n-2}/\Bbb Z^{2n-2}$, and $\theta_{[\lambda]}(S;\tau,0)$ is a theta series of $2n-2$ variables.
Then
$$\mathcal F_{f,g}(Z)=\int_{\Gamma_0(4)\backslash \Bbb H} \left(\sum_S \mathcal F_S(\tau,0) e^{2\pi i Tr(ZS)}\right)
\overline{g(\tau))}(Im \tau)^{k+n-\frac 52}\, d\tau
=\sum_S A_S e^{2\pi i Tr(ZS)},
$$
where 
$$A_S=\int_{\Gamma_0(4)\backslash \Bbb H} \mathcal F_S(\tau,0)\overline{g(\tau)} (Im \tau)^{k+n-\frac 52}\, d\tau.
$$
Now we consider the special case: $2n-2$ is a multiple of $8$, and choose $S$ so that $2S$ is an even unimodular matrix. Then $\Lambda$ is trivial, and $\mathcal F_S(\tau,0)=H(\tau)\theta_S(\tau)$, where $H(\tau)$ is a cusp form of weight $k+\frac 12$.

Now we apply the Siegel formula.
For $8|(2n-2)$, let $C_{2n-2}$ be the set of classes (up to isomorphism) of even unimodular lattices of $2n-2$ variables, and let $g_S\in C_{2n-2}$ be the order of the automorphism group of $S\in C_{2n-2}$. Then 
$$\sum_{S\in C_{2n-2}} \frac 1{g_S} \theta_S=M_{2n-2} E_{\frac {2n-2}2},\quad \text{where $M_{2n-2}=\sum_{S\in C_{2n-2}} \frac 1{g_S}$}.
$$
Now consider the sum
$$\sum_{S\in C_{2n-2}} \frac 1{g_S} A_S.
$$
It is
\begin{equation}\label{half-1}
M_{2n-2}\int_{\Gamma_0(4)\backslash \Bbb H} H(\tau)\overline{g(\tau)} E_{n-1}(\tau) (Im \tau)^{k+n-\frac 52}\, d\tau.
\end{equation}
It may be possible to show that it is non-vanishing. Then one of $A_S$ is non-vanishing, and $\mathcal F_{f,g}$ is non-vanishing.

\end{document}